\documentclass[a4paper,11pt]{amsart}
\usepackage[utf8]{inputenc}
\usepackage{amssymb, mathrsfs, amsfonts, amsmath}
\usepackage{mathtools} 
\usepackage{hyperref}
\usepackage{amsthm}
\usepackage{tikz}
\usepackage[english]{babel}
\usepackage{xcolor} 
\usepackage{geometry}
\usepackage{lipsum}

\title[Hilbert transform along the parabola]{The Hilbert transform along the parabola, the polynomial Carleson theorem and oscillatory singular integrals}
\author{Jo\~ao P. G. Ramos}
\address{Mathematisches Institut der Universit\"at Bonn, Endenicher Allee 60, 53115 Bonn, Germany}
\email{joaopgramos95@gmail.com} 

\newcommand{\R}{\mathbb{R}}
\newcommand{\Z}{\mathbb{Z}}

\newcommand{\mmd}{\mathrm{d}}

\newtheorem{theorem}{Theorem}

\newtheorem{lemma}{Lemma}
\newtheorem{prop}{Proposition}

\newtheorem{quest}{Question}

\def\XXint#1#2#3{{\setbox0=\hbox{$#1{#2#3}{\int}$}
\vcenter{\hbox{$#2#3$}}\kern-.5\wd0}}

\begin{document}
\begin{abstract}
We make progress on an interesting problem on the boundedness of maximal modulations of the Hilbert transform along the parabola. Namely, if we consider the multiplier
arising from it and restrict it to lines, we prove uniform $L^p$ bounds for maximal modulations of the associated operators. Our methods consist of identifying where to 
use effectively the polynomial Carleson theorem, and where we can take advantage of the presence of oscillation to obtain decay through the $TT^*$ method.
\end{abstract}
\maketitle

\section{Introduction}

\subsection{Historical background} We define the Hilbert transform along the parabola as 
\begin{equation}\label{riesz}
\mathcal{H}_2f(x,y) = \text{p.v.} \int_{\R} f(x-t,y-t^2) \frac{\mmd t}{t},
\end{equation}
where we let $f \in \mathcal{S}(\R^2).$ This operator has an anisotropic symmetry and has been considered in the wider framework of anisotropically homogeneous operators, dating back to the work of Fabes and Rivi\`ere \cite{FabesRiviere} in the 1970's.
In this context, $L^p$ estimates for such operators imply additional regularity of solutions of certain associated parabolic partial differential equations. \\ 

For the particular case of the Hilbert transform along the parabola, the works of Nagel, Rivi\`ere and Wainger \cite{NRW1, NRW2} prove that it is indeed bounded in $L^p(\R^2)$. Their results
provide, in fact, $L^p-$bounds for higher dimensional generalizations of this operator. Possible generalizations have been further explored in the nilpotent groups case by Christ \cite{MChrist}, as well as the question 
of weak-type endpoint estimates in the work of Christ and Stein \cite{ChristStein} and, more recently, in the work of Seeger, Tao and Wright \cite{SeegerTaoWright}. \\ 

Parallelly to that, the theory of \emph{maximally modulated} Calder\'on-Zygmund operators also developped in the last 50 years. Indeed, in 1966, in order to prove almost everywhere convergence of Fourier 
series in $L^2,$ Carleson \cite{Carleson} considers the operator 
\[
Cf(x):= \sup_{N \in \R} \left| \int_{\R} f(x-t) e^{iNt} \, \frac{\mmd t}{t}\right| = \sup_{N \in \R} |H(e^{iN(\cdot)} f)|(x). 
\]
This is now called the \emph{Carleson operator}. After Carleson's paper, many works have been dedicated to sharpening and perfecting
his proof. Hunt \cite{Hunt} extended, in 1967, Carleson's result to all $L^p$ spaces, $p \in (1,+\infty),$ and Fefferman \cite{Fefferman} and Lacey and Thiele \cite{LaceyThiele} provided different proofs of the same result. All of the proofs 
above share, however, the property of employing a time-frequency decomposition to encompass translation, dilation and modulation symmetries of the Carleson operator. \\

Inspired by that result, E. M. Stein \cite{Stein2} posed the following problem: if instead of linear phases, we take suprema over \emph{polynomial} phases, do we still have $L^p$ bounds? Namely, if one considers the operator 
\[
f \mapsto \sup_{\text{deg }P \le n} \left| \int_{\R} f(x-t) e^{iP(t)} \, \frac{\mmd t}{t}\right|,
\]
is it bounded in $L^p(\R), p \in (1,+\infty)?$ A first step in this direction is the work of Stein and Wainger \cite{SteinWainger}, where they consider a restricted supremum over polynomials \emph{without} the linear term.
Unlike the proofs of bounds for the Carleson operator, this does not rely on a time-frequency decomposition directly, but on a dyadic decomposition and $TT^*$ method to exploit oscillatory integral estimates. \\ 

In subsequent works, Lie \cite{Lie1} treated the case of weak-type $(2,2)$ bounds for the operator above if $n=2,$ and considered in \cite{Lie2} the general $n \ge 1$ case, in the one dimensional setting. 
More recently, Zorin-Kranich \cite{pavel} extended the analysis of the operator above for \emph{higher} dimensions and Calder\'on-Zygmund operators with fairly general conditions. Their techniques, however, resort more to time-frequency methods 
in the style of Fefferman \cite{Fefferman} rather than the $TT^*$ strategy of Stein and Wainger. \\ 

Pierce and Yung \cite{PierceYung} considered a hybrid version of the two parallel kinds of results we discussed. In particular, they consider operators of the form 
\[
f(x,y) \mapsto \sup_{P \in \mathcal{P}} \left| \int_{\R^d} f(x-t,y-|t|^2) e^{iP(t)} K(t) \, {\mmd t} \right|,
\]
where $K$ is a suitable Calder\'on-Zygmund kernel and $\mathcal{P}$ some finite-dimensional subspace of polynomials. They obtain $L^p$ estimates for certain subspaces that avoid linear and some quadratic terms, as long as $d \ge 2$. Subsequently to it, Guo, Pierce, 
Roos and Yung \cite{GPRY} considered the $d=2$ case by taking a partial supremum for curves like $(t,t^d)$ and $P(t) = N\cdot t^m$. For these results, as well as the ones in \cite{PierceYung}, the strategy resembles that of Stein and Wainger, in the sense that the main 
tools are still dyadic decompositions, $TT^*$ estimates and suitable oscillatory integral estimates to obtain decay. \\ 

Continuing this line of thought, the following question arises naturally in \cite{GPRY}: 

\begin{quest}\label{parabcarl} For $f \in \mathcal{S}(\R^2),$ is the \emph{parabolic} Carleson operator 
\begin{equation}
\mathcal{C}_2f(x,y) := \sup_{N,M \in \R} \left| \int_{\R} f(x-t,y-t^2) e^{iNt + i Mt^2} \, \frac{\mmd t}{t}\right|
\end{equation}
bounded in $L^2(\R^2)?$ 
\end{quest} 

This is nothing but a supremum of the Hilbert transform along the parabola of all possible modulations of $f.$ In other terms, this operator admits a representation as 
\[ 
\sup_{N,M}|\mathcal{H}_2(e^{iN(\cdot)_1 + iM(\cdot)_2}f)(x,y)| = \sup_{N,M}|\mathcal{F}^{-1} ( m_2(\cdot+(N,M)) \widehat{f})(x,y)|,
\]
where $m_2\left(\frac{\xi}{2\pi},\frac{\eta}{\sqrt{2\pi}}\right) = \int_{\R} e^{i (\xi \cdot t + \eta \cdot t^2)} \frac{\mmd t}{t},$ and $\mathcal{F}f(\xi) = \widehat{f}(\xi) = \int_{\R^2} f(y) e^{-2 \pi i \xi \cdot y} \, \mmd y$  denotes the Fourier transform. Here and henceforth 
we abuse notation and denote by $m_2$ the dilation given above by $m_2(\cdot/2\pi, \cdot/\sqrt{2\pi}).$ \\ 

Partial progress in Question \ref{parabcarl} has been made by Roos \cite{Roos}, where he extends the techniques from Lacey and Thiele \cite{LaceyThiele} to the anisotropic case. The main obstacle to apply his result to Question \ref{parabcarl} is the fact that
$m_2$ is only H\"older continuous of exponent $<1$ along $\R,$ while Roos needs regularity of his multipliers greater than three times the anisotropic degree of homogeneity. Interestingly enough, Zorin-Kranich mentions in \cite{pavel} that it should be possible to extend his 
results on the polynomial Carleson operator to the anisotropic context. This would, however, still not imply the validity of Question \ref{parabcarl}, as the techniques from Zorin-Kranich only yield bounds for symbols with regularity at least equal to the homogeneity degree. 

\subsection{Main results} We are interested in the restriction of $m_2(\xi,\eta)$ to lines. That is, we consider the family of one-dimensional functions given by $m_{a,b}(\eta) = m_2(a\eta + b, \eta), a,b \in \R.$ In order to consider also horizontal lines, we define $m_{+\infty,b}(\eta) := m_2(\eta,b).$  They define, via Fourier inversion, a family of linear operators 
$T_{a,b}f(x) := \mathcal{F}^{-1}(m_{a,b}\widehat{f})(x)$ in dimension 1. Our main result deals with maximal modulations of these operators -- or, analogously, maximal translations in the multiplier side: 

\begin{theorem}\label{lines} Let $\mathcal{C}_{a,b}f(x) = \sup_{N \in \R} |T_{a,b}(e^{i N \cdot} f)(x)|.$ Then it holds that 
\[
\sup_{a \in \R \cup \{+\infty\}} \| \sup_{b \in \R} \mathcal{C}_{a,b}\|_{p \to p} < +\infty, \; \forall p \in (1,+\infty).
\]
\end{theorem}

Intuitively, Theorem \ref{lines} represents taking a very thin strip around the line $(a \eta + b, \eta)$ and a function with Fourier transform supported there and calculating $\mathcal{C}_2$ with this additional restriction. \\

It is not hard to see that Theorem \ref{lines} follows if the answer to Question \ref{parabcarl} is affirmative. Indeed, for any line $\ell \subset \R^2$ as above consider a strip $S_{\delta}$ of width $\delta >0$ with direction $\ell.$ Consider also the set of lines 
$\ell' \sim \ell$ parallel to $\ell.$ If we consider functions $F_{\delta}$ such that their Fourier transform is supported on $S_{\delta},$ is essentially constant along the perpendicular direction to $\ell$ and equals $\widehat{g}$ on $\ell$,  we have, formally, 
\begin{equation}\label{implication}
\|\sup_{\ell'\sim \ell} \mathcal{C}_{\ell'}g\|_{L^2(\R)} \le \lim_{\delta \to 0} \|\mathcal{C}_2(F_{\delta})\|_2 \le C \limsup_{\delta \to 0} \|F_{\delta}\|_2 = C \|g\|_{L^2(\R)}.
\end{equation}
We elaborate more on \eqref{implication} in Section \ref{comments}. \\

Our first task is to pass from the rather complicated formulation in Theorem \ref{lines} to a formulation with which we can work more directly. This is the main content of the next proposition, 
which we prove in Section \ref{reduction}.

\begin{prop}\label{operss} Let $[u]^{1/2}$ denote either $|u|^{1/2}$ or $\text{sign}(u)|u|^{1/2}.$ Suppose that the operators 
\begin{equation}\label{oper}
\mathfrak{C}^Rf(x) = \sup_{N,b \in \R} \left| \int_{-R}^{R} f(x-t) e^{iNt} e^{ib[t+1]^{1/2}} \, \frac{\mmd t}{t} \right| 
\end{equation}
are bounded in $L^p$, for $1 < p < +\infty,$ independently of the truncating parameter $R>0.$ Then Theorem \ref{lines} holds.  
\end{prop}

We still have to bound the operator arising from the proposition above. The following result asserts the boundedness of the second maximal operator. 

\begin{theorem}\label{main} Let $\mathfrak{C}^R$ be defined as \eqref{oper} above. It holds that
\[
\|\mathfrak{C}^Rf\|_p \le C_p \|f\|_p,
\]
for all $f \in L^p(\R)$ and all $p \in (1,+\infty),$ and $C_p$ \emph{independent} of $R>0.$ 
\end{theorem} 

The proof of Theorem \ref{main} is the main novelty of this article. In order to prove it, we employ two different ideas. More specifically, we first prove that we can regard 
the parameter $b$ as belonging to a fixed dyadic scale $\sim 2^k$, as long as we prove summable decay in $|k|.$ We then break the interval of integration defining $\mathfrak{C}^Af$ into distinct 
regimes of intervals, namely mainly the ones where oscillatory behaviour is strong enough to enable the use of the $TT^*$ method, and the ones where the phase is mimics a polynomial Carleson operator, as considered
in \cite{Lie2,pavel}. 

The crucial point of using the polynomial Carleson theorem together with an application of $TT^*$ to prove Theorem \ref{main} is that this technique is optimal. That is, for smaller scales, the oscillatory integral estimates used in the $TT^*$ method only give us a bound not decaying with $b \sim 2^k$. Truncating at a high power of $b$ 
does not appear randomly. On the other hand, one asks whether it is possible to use a comparison to a polynomial Carleson operator \emph{directly} at least in the interval $[-1/2,1/2].$ We finish our discussion of the proof of Theorem \ref{lines} by proving that it is impossible unless
letting the degree of the polynomial tend to infinity. \\

\subsection{Notation} Some remarks are in order to facilitate the reader's understanding: 
\begin{enumerate} 
\item We denote by $C>0$ a constant that may change from line to line; 
\item We write throughout the paper $A \lesssim B$ to mean that $A \le C \cdot B,$ for some constant $C.$ If $C$ depends on some parameter $\delta$ in a relevant way, we write $A \lesssim_{\delta} B;$ 
\item $\psi$ generally denotes positive bump functions with some partition of unity property, whereas $\phi,\varphi$ denote usually phase functions in oscillatory integrals; 
\item Finally, unless otherwise stated, we consider the functions in the proofs below to belong to $\mathcal{S}(\R)$ and extend the respective bounds by density.
\end{enumerate} 

\section{Proof of Proposition \ref{operss}}\label{reduction} In order to prove Theorem \ref{lines}, we first reduce the analysis to simpler operators. We mention that, if we let go of the uniformity of the estimates on the line, 
the comparisons in this section can be made much looser. Therefore, the emphasis is on getting bounds \emph{independent} on the paramers when comparing.

\subsection{From two to one parameter} We reduce the task of proving Theorem \ref{lines} to the analysis of a one-parameter family of operators. 
We must consider \emph{all} lines in the analysis, and therefore also $m_{+\infty,b}(\xi) = m_2(\xi,b)$ must be considered as a multiplier. Nevertheless, an analysis 
identical to the one undertaken below shows that the operator 
\[
\sup_{b,N} |\mathcal{F}^{-1} (m_{+\infty,b} \widehat{(\mathcal{M}_N f)})|
\]
is just a quadratic Carleson operator, so, by the results in \cite{Lie1,Lie2,pavel},
we need not include it in our discussion. 
The main tool to our reduction will be the following Lemma: 

\begin{lemma}\label{trivial} Let $\{g_{a,b}\} \subset L^1(\R)$ be a family of positive functions with $|g_{a,b}| \le h_{a}$ pointwise for another family $\{h_{a}\},$ which 
is uniformly bounded in $L^1.$ I.e., $\sup_{a} \|h_{a}\|_1 < +\infty.$ It holds that 
\[
\sup_{a}\| \sup_{b,N} |g_{a,b} * (\mathcal{M}_Nf)| \|_p \le C \|f\|_p,
\]
 with $\mathcal{M}_Nf(x) = e^{2\pi iNx}f(x).$
\end{lemma}
\begin{proof} By Young's convolution inequality, 
$$\|\sup_{b,N} |g_{a,b} * (M_Nf)|\|_{L^p} \le \| h_{a} * |f| \|_{L^p} \le \|h_{a}\|_1 \|f\|_{L^p} \le \left(\sup_{a} \|h_{a}\|_1\right) \|f\|_{L^p} =: C\|f\|_{L^p}.$$
\end{proof}

We will especially use it in the following form: if two families of maximally modulated operators $\mathcal{O}^{a,b}_if(x) := \sup_N |O_i^{a,b}(\mathcal{M}_Nf)(x)|,\,i=1,2$ satisfy 
\[
|\mathcal{O}^{a,b}_1f(x) - \mathcal{O}^{a,b}_2f(x)| \le |f| * h_{a}(x),
\]
with $h_a \in L^1(\R)$ as in Lemma \ref{trivial} above, then bounding $\mathcal{O}^{a,b}_1$ in $L^p$ uniformily in $a$ is equivalent to bounding $\mathcal{O}^{a,b}_2$ uniformily in $a.$ With this in mind, we rewrite our multipliers as 
\begin{align*}
m_2(2a \eta + b, \eta) = & \int_{\R} e^{i(2a \eta t + b t + \eta t^2)} \frac{\mmd t}{t} = e^{i\cdot a^2 \cdot \eta}\int_{\R} e^{i(t-a)^2\eta + ibt} \, \frac{\mmd t}{t}, \cr
\end{align*}
We further rewrite the operators $T_{2a,b}$ using Fourier inversion:
\[
T_{2a,b}f(x) = \int_{\R} f(x + a^2 - (t-a)^2) e^{ibt} \frac{\mmd t}{t}.
\]
Notice that the $a^2$ term only contributes as a translation in the $x-$variable, so we consider the simpler operators
\[ 
\tilde{T}_{2a,b}f(x) := \int_{\R} f(x - (t-a)^2) e^{ibt} \frac{\mmd t}{t}.
\]
By translating $t$ by $a$ and changing variables $s = t^2$ -- after breaking the integral into $\R_{+}$ and $\R_{-}$ parts --, we get from the observation above that bounding $\sup_{N,b}|\tilde{T}_{2a,b}(\mathcal{M}_Nf)|$ is equivalent
to bounding 
\[
A_{a}f(x) = \sup_{N,b}\left|\int_{0}^{+\infty} (\mathcal{M}_Nf)(x-s) \left(\frac{e^{ibs^{1/2}}}{s^{1/2}(s^{1/2} - a)} - \frac{e^{-ibs^{1/2}}}{s^{1/2}(s^{1/2} + a)} \right) \mmd s\right|. 
\]
\subsection{Reduction to model operators} We now look more closely into this family of operators. First, we rewrite the kernel defining $A_{a}$ as 

\begin{align*}  
\left(\frac{e^{ibs^{1/2}}}{s^{1/2}(s^{1/2} - a)} - \frac{e^{-ibs^{1/2}}}{s^{1/2}(s^{1/2} + a)} \right) & = \frac{1}{2 s^{1/2}} \cdot \frac{e^{ibs^{1/2}} - e^{-ibs^{1/2}}}{s^{1/2} + a} + \frac{a}{s^{1/2}} \cdot \frac{e^{ibs^{1/2}}}{s-a^2}  \cr
 & =: K_{1,a}^b(t) + K_{2,a}^b(t). \cr
\end{align*}

If $a =0,$ we have $K_{2,0}^b(t) = 0$, whereas $K_{1,0}^b(t)$ becomes $$\frac{e^{ibt^{1/2}} - e^{-ibt^{1/2}}}{2t}.$$ 
We write, for the time being, the maximal operator we are left with as 
\begin{equation}\label{op11}
\sup_{N,b} \left| \int_{0}^{+\infty} (\mathcal{M}_Nf)(x-t) \frac{e^{ibt^{1/2}} - e^{-ibt^{1/2}}}{2t} \, \mmd t \right|.
\end{equation}

For the $a \ne 0$ cases, we notice that $A_{-a}(f) = {A_{a}f},$ so we suppose without loss of generality that $a>0.$ We bound the kernel pointwise in a suitable neighborhood of the origin, 
and compare it to another operator away. Especifically, for $0 \le t \le a^2, $ we have 
\[
|K_{1,a}^b(t)| \le \min(1,t^{-1/2}) \cdot \frac{1}{t^{1/2} + a}.
\]
It is then easy to see that $\int_0^{a^2} \min(1,t^{-1/2}) \cdot \frac{1}{t^{1/2} + a} \mmd t$ is bounded independently of $a >0.$ By Lemma \ref{trivial}, we are left with the $t \ge a^2$ portion, where we estimate 
\[ 
\left| K_{1,a}^b(t) - \frac{e^{ibt^{1/2}} - e^{-ibt^{1/2}}}{2t} \right| \le \frac{a}{t(t^{1/2} + a)}.
\]
It is again straightforward to see that $\int_{a^2}^{\infty} \frac{a}{t(t^{1/2}+a)} \mmd t < 10, \, \forall a >0.$ We have again reduced the analysis to the operator in \eqref{op11}. 

We now address the $K_{2,a}^b$ part. We split the integral defining $K_{2,a}^b *(\mathcal{M}_Nf)$ into three regimes: $[0,+\infty)=[0,a^2/2]\cup(a^2/2,3a^2/2) \cup [3a^2/2,+\infty).$ For each of them, we have: 

\begin{itemize} 
\item $|K_{2,a}^b(t)| \le \frac{4}{a^2}$ for $t \in [0,a^2/2].$ Therefore, convolution with this part can be controlled by Lemma \ref{trivial};
\item For $t \ge 3a^2/2,$ we estimate $|K_{2,a}^b(t)| \le \frac{a}{t^{1/2}(t-a^2)},$ where a change of variables leads us to conclude that $\int_{3a^2/2}^{+\infty} \frac{a}{t^{1/2}(t-a^2)} \mmd t < 10$;
\item For the singular middle interval, we compare:
\[ 
\left|K_{2,a}^b(t) - \frac{e^{ibt^{1/2}}}{t-a^2}\right| \le \frac{1}{t^{1/2}(t^{1/2} + a)}. 
\]
By a change of variables $t \mapsto s^2$, one sees that $\int_{a^2/2}^{3a^2/2} \frac{\mmd t}{2t^{1/2}(t^{1/2} + a)} = \log \left(\frac{\sqrt{3/2} +1}{\sqrt{1/2} + 1}\right).$ By observing the operator to which we compared, we conclude that it is enough to control the $L^p$ norm of
$\sup_{N,b} \left| \int_{-a^2/2}^{a^2/2} f(x-t) e^{iNt} e^{ib\sqrt{t+a^2}} \frac{\mmd t}{t} \right|$ independently of the parameter $a.$ changing variables, it is easy to see that this expression equals $\mathfrak{C}^{1/2}(f_a)(x/a^2),$ where $f_{a^2}(y) = f(a^2y).$ If Theorem \ref{main} holds, 
then the $L^p$ norm of this expression is bounded independently of $a.$ 
\end{itemize}

In order to finish the proof of Proposition \ref{operss}, we must conclude boundedness of the operator 
\[ 
\sup_{N,b} \left| \int_{0}^{+\infty} (\mathcal{M}_Nf)(x-t) \frac{e^{ibt^{1/2}} - e^{-ibt^{1/2}}}{2t} \, \mmd t \right|
\]
given Theorem \ref{main}. In fact, we first need an auxiliary result: 

\begin{prop}\label{consequence} Suppose Theorem \ref{main} holds. Then the maximal functions 
\[
f \mapsto \sup_{N,b} \left| \int_{\R} f(x-t) e^{iNt}e^{ib[t]^{1/2}} \, \frac{\mmd t}{t} \right|
\]
are both bounded in $L^p.$ 
\end{prop}

\begin{proof} We consider $\mathfrak{C}^{+\infty} := \mathfrak{C}$ as in Theorem \ref{main}. If we define 
\[
\mathfrak{C}_af(x) = \sup_{N,b} \left| \int_{\R} f(x-t) e^{iNt} e^{ib[t+a]^{1/2}} \, \frac{\mmd t}{t}\right|,
\]
the dilation symmetries of $\mathfrak{C}$ imply that $\mathfrak{C}f(x) = \mathfrak{C}_a(f_{1/a})(ax),$ where $f_{1/a}(y) = f(y/a).$ This plainly implies 
\begin{equation}\label{unifff}
\| \mathfrak{C}_af\|_p = a^{1/p} \|\mathfrak{C}(f_{a})\|_p \lesssim a^{1/p} \|f_{a}\|_p = \|f\|_p.
\end{equation}
Now, a direct computation together with dominated convergence shows that 
\[
\liminf_{a \to 0} \mathfrak{C}_af(x) \ge \sup_{N,b} \left| \int_{\R} f(x-t) e^{iNt}e^{ib[t]^{1/2}} \, \frac{\mmd t}{t} \right|
\]
for all smooth $f$ with compact support. The proposition is then proved by Fatou's lemma and \eqref{unifff}. 
\end{proof}

We rewrite the operators from Proposition \ref{consequence} as 
\[
\sup_{N,b} \left| \int_{0}^{+\infty} ((\mathcal{M}_Nf)(x-t) e^{ib[t]^{1/2}} - (\mathcal{M}_Nf)(x+t) e^{ib[-t]^{1/2}}) \, \frac{\mmd t}{t} \right|
\]
If $[u]^{1/2} = \text{sign}(u) |u|^{1/2}$, the integrand equals 
$$(\mathcal{M}_Nf)(x-t) e^{ibt^{1/2}} - (\mathcal{M}_Nf)(x+t) e^{-ibt^{1/2}},$$
and it becomes $((\mathcal{M}_Nf)(x-t) - (\mathcal{M}_Nf)(x+t)) e^{ibt^{1/2}}$ in case $[u]^{1/2} = |u|^{1/2}.$ Notice now that 
\begin{align*}
(\mathcal{M}_Nf)(x-t) \cdot (e^{ibt^{1/2}} - e^{-ibt^{1/2}}) & = (\mathcal{M}_Nf)(x-t) e^{ibt^{1/2}} - (\mathcal{M}_Nf)(x+t) e^{-ibt^{1/2}} \cr 
 & - ((\mathcal{M}_Nf)(x-t) - (\mathcal{M}_Nf)(x+t)) e^{-ibt^{1/2}},\cr
\end{align*}
so that the operator from \eqref{op11} is bounded by the sum of the two in Proposition \ref{consequence}. This finishes the reduction to Theorem \ref{main}.

\section{Proof of Theorem \ref{main}}\label{modeluni} 

In order to deal with the operators $\mathfrak{C}^R,$ we use the Kolmogorov-Seliverstov linearization method. In fact, by suitably choosing, we find two functions $b,N: \R \to \R_{+}$, taking on only finitely many values, so that 
\[
|\mathfrak{C}_{b,N}^R f(x)| = \left| \int_{-R}^{R} f(x-t) e^{ib(x)[t+1]^{1/2}} e^{i N(x) t} \frac{\mmd t}{t} \right| \ge \frac{1}{2} \mathfrak{C}^Rf(x). 
\]
Our goal is to bound this operator \emph{independently} of both $b$ and $N$, as well as $R>0$. We omit therefore $b,N,R$ in order to clean up notation. The first step is to split the analysis of this operator into two parts:\\ 
\begin{align*}
|\mathfrak{C} f(x)| & \le 1_{\{b(x) \le 10\}} \mathfrak{C}f(x) + 1_{\{b(x) > 10\}} \mathfrak{C}f(x) \cr 
 & =: \mathfrak{C}^1f(x) + \mathfrak{C}^2f(x).
\end{align*}

\noindent\textbf{Part 1: Analysis of $\mathfrak{C}^1$.} We split the interval $[-R,R]$ of the integral defining $\mathfrak{C}^1$ as 
\[
[-R,-\min\{R,b(x)^{-2}\}] \cup (-\min\{R,b(x)^{-2}\},\min\{R,b(x)^{-2}\}) \cup [\min\{R,b(x)^{-2}\},R].
\]
In the middle interval, the aim is to simply approximate the phase $b(x)[t+1]^{1/2}$ by $b(x).$ In more effective terms, the difference
\[
\left|\int_{-\min\{R,b(x)^{-2}\}}^{\min\{R,b(x)^{-2}\}} f(x-t) e^{iN(x)t} e^{ib(x)[t+1]^{1/2}} \, \frac{\mmd t}{t} - e^{ib(x)} \int_{-\min\{R,b(x)^{-2}\}}^{\min\{R,b(x)^{-2}\}} f(x-t) e^{iN(x)t} \, \frac{\mmd t}{t} \right|
\]
is bounded pointwise by 
\[
b(x) \int_{-\min\{R,b(x)^{-2}\}}^{\min\{R,b(x)^{-2}\}} |f(x-t)||[t+1]^{1/2} -1| \, \frac{\mmd t}{|t|}. 
\]
Notice that the difference $h(t) = |[t+1]^{1/2} - 1|$ satisfies that 
$$h(t) \le \begin{cases} 4|t|, & \text{if } t \in [-1/2,2]; \\ 4, & \text{if } t \in [-2,-1/2]; \\ 4|t|^{1/2}, & \text{if } |t| \ge 2. \end{cases}$$
The function $h(t)/|t|$ admits then a radial majorizer $H(t)$ whose integral is at most a multiple of $\min\{R,b(x)^{-2}\}^{1/2}.$ Because of the multiplying $b(x)$ factor in front, 
this integral is pointwise bounded by an absolute constant times the Hardy-Littlewood maximal function of $f$ at the point $x.$ It is well-known (cf. Grafakos \cite[Section~6.3]{grafakosmodern}) 
that the maximally truncated version of the Carleson operator is bounded in $L^p.$ The $L^p$ boundedness for $\mathfrak{C}^1$ restricted to this middle interval then follows 
from $L^p$ boundedness of the maximal function. \\ 

For the two outer intervals, the main idea is to use the $TT^*$ method to get summable decay in the scales. This is a crucial idea in this argument, and this will be emphasized by its incidence in this section. \\

Namely, we suppose that $R> b(x)^{-2}$, as this part of the analysis gets trivialized in case $R \le b(x)^{-2}.$ Let $\psi_0 : \R \to \R$ be a positive, smooth bump function supported in $[1/2,2]$ such that 
\[
\sum_{j \in \Z} \psi_0 \left( \frac{\xi}{2^j}\right) \equiv 1, \, \forall \xi \in \R \setminus \{0\}.
\]
We analyze the integral only over the interval $[b(x)^{-2},R],$ as the other part the analysis is entirely analogous. By a computation analogous to the one performed above to control the middle interval, 
we obtain that the integral defining this operator over the interval $[b(x)^{-2},R]$ is, modulo error terms amounting to maximal function, equal to 
\begin{equation}\label{decomp1}
\sum_{j \ge 0} \int_{\R} f(x-t) e^{iN(x)t} e^{ib(x)[t+1]^{1/2}} \psi_0(2^{-j}b(x)^2 t) \, \frac{\mmd t}{t} =: \sum_{j \ge 0} \mathfrak{S}^jf(x). 
\end{equation}
Some remarks are in order about the operators $\mathfrak{S}^j.$ First of all, these operators are pointwise bounded by an absolute constant times the Hardy--Littlewood maximal function, due to the space localization $1/2 \le 2^{-j} b(x)^2 t \le 2$ 
imposed in the integral. Therefore, we immediately get 
\[
\|\mathfrak{S}^jf\|_{1,+\infty} \lesssim \|f\|_1, \, \, \|\mathfrak{S}^jf\|_{\infty} \lesssim \|f\|_{\infty}.
\]
In order to conclude bounds on the sum in \eqref{decomp1}, it suffices to prove that $\|\mathfrak{S}^jf\|_{2} \lesssim 2^{-\tau j} \|f\|_2,$ for some $\tau >0.$ Indeed, by interpolating with the 
endpoint estimates above, we obtain that there is $\tau_p > 0$ such that 
\[
\|\mathfrak{S}^j f\|_p \lesssim_p 2^{-\tau_p j} \|f\|_p, \, \forall p \in (1,+\infty).
\]
Finally, the $L^p-$norm of the expression in \eqref{decomp1} is controlled, by triangle inequality, by 
\[
\sum_{j \ge 0} \|\mathfrak{S}^jf\|_p \lesssim_p \sum_{j\ge 0} 2^{-\tau_p j} \|f\|_p \lesssim_p \|f\|_p.
\]
We focus hence on the extra decay for the $L^2$ bounds. That is the content of the following proposition. 

\begin{prop}\label{propdec} Let $\mathfrak{S}^j$ be defined as above. It holds that
\[
\|\mathfrak{S}^jf\|_2 \lesssim 2^{-j/200} \|f\|_2,
\]
for all $f \in L^2(\R).$
\end{prop}

\begin{proof} We first write the operator $\mathfrak{S}^jf(x) = S^j_{N(x),b(x)} * f(x),$ where we define 
\[
S^j_{N(x),b(x)}(s) = e^{iN(x)s} e^{ib(x)[s+1]^{1/2}} \psi_0(2^{-j}b(x)^2s) \, \frac{1}{s}. 
\]
Now, in order to compute the $L^2$ norm of $\mathfrak{S}^j,$ we compute instead its composition with its adjoint. It admits an expansion as 
\[
\mathfrak{S}^j (\mathfrak{S}^j)^* f(x) = \int_{\R} (S^j_{N(x),b(x)}*\tilde{S}^j_{N(y),b(y)})(x-y) f(y) \, \mmd y.
\]
Here, we let $\tilde{S}^j_{N(x),b(x)}(z) = \overline{S^j}_{N(x),b(x)}(-z).$ A computation shows, on the other hand, that $| (S^j_{N(x),b(x)}*\tilde{S}^j_{N(y),b(y)})(\xi)|$ equals 
\[
\left| \int_{\R} e^{i(N(x)-N(y))s} e^{i(b(x)[s+1]^{1/2} - b(y)[s-\xi + 1]^{1/2})} \frac{\psi_0(2^{-j}b(x)^2s)}{s} \frac{\psi_0(2^{-j}b(y)^2(s-\xi))}{s-\xi} \, \mmd s\right|.
\]
We change variables in this last integral to simplify the analysis. Effectively, assume, without loss of generality, that $b(y) < b(x).$  We let $s' = 2^{-j} b(x)^2 s,$ and denote $\xi' = 2^{-j}b(y)^2 \xi.$ The integral whose absolute value we would 
like to estimate rewrites then as 
\[
\frac{b(y)^2}{2^j} \int_{\R} e^{i2^jb(x)^{-2}(N(x)-N(y))s'} e^{iR_{\xi',j,b}(s')} \frac{\psi_0(s')}{s'} \frac{\psi_0(hs'- \xi')}{hs' - \xi'} \, \mmd s',
\]
where we let $h := \frac{b(y)^2}{b(x)^2} < 1,$ and consider the phase function given by 
\[
R_{\xi',j,b}(s') = 2^{j/2} ([s' + 2^{-j} b(x)^2]^{1/2} - [hs' - \xi' + 2^{-j} b(y)^2]^{1/2}).
\]
This is the oscillatory integral we would like to estimate. The following lemma is the tool to directly do it. 
\begin{lemma}\label{oscilest0} Let $\Psi: \R \times \R \to \R$ be a smooth function supported in $\{ (s',\xi') \in \R^2 \colon s', hs' - \xi' \in [1/2,2]\},$ for some fixed positive parameter $h \le 1.$ It holds that, for all $v, \xi' \in \R$ and $j \ge 0,$  
\begin{align}\label{oscil0}
 & \left| \int_{\R} e^{i v s'} \cdot e^{i R_{\xi',j,b}(s')}  \Psi(\xi',s') \mmd s' \right|   \cr
 \lesssim \sup_{\xi' \in [-2,2]} \|\Psi(\xi',\cdot)\|_{C^2} & \left( 1_{[-2^{-j/100},2^{-j/100}]}(\xi') +  2^{-\frac{j}{9}} 1_{[-4,4]}(\xi')\right), \cr
\end{align}
where the implicit constant does not depend on $h \in (0,1].$ 
\end{lemma}

\begin{proof} The proof follows the essential principle that, for $\xi'$ small enough, we cannot expect much more from the integral than the trivial triangle inequality bound, and if $\xi'$ is non-small, the oscillation of the phase 
$R_{\xi',j,b}$ starts providing cancellation, and therefore decay. In fact, if $|\xi'| \le 2^{-j/100},$ we use triangle inequality as pointed out, and one readily obtains the first term on the right 
hand side of the statement. \\ 

If, nonetheless, $|\xi'| \ge 2^{-j/100},$ we have to prove some sort of lower bound on the derivatives of the (completed) phase 
\[
\varphi(s') = vs' + R_{\xi',j,b}(s').
\]
For that purpose, we consider the vector 
\[
\mathcal{Q}(s') = \begin{pmatrix} \varphi''(s') \\ -\frac{2}{3} \varphi'''(s') \end{pmatrix}
\]
of second and third derivatives of the phase. The aim is to prove this is bounded from below by some positive power of $2^j$, so that stationary phase considerations give us the desired decay. In order to prove this bound, we adopt a strategy already present in 
\cite{GPRY} and \cite{GHLR}. Namely, the idea used multiple times there is to rewrite the vector $\mathcal{Q}$ as a certain (invertible) matrix applied at a vector. If we prove sufficiently good bounds on the determinant and norms of the objects 
involved, we should get good enough bounds on the original $\mathcal{Q}.$ A more precise version of this principle is the following Lemma. 

\begin{lemma}\label{matrix} Let $A$ be a $n \times n$ invertible matrix. It holds that 
\[
|A \cdot x| \ge |\det (A)| \|A\|^{1-n} \cdot |x|.
\]
\end{lemma} 

\begin{proof} Assume, by homogeneity, that $\|A\|=1.$ Let us show first that $\|A^{-1}\| \le \frac{1}{\det(A)}.$  It is simple to see that the eigenvalues of $AA^{*}$ are all contained in $[0,1]$. If $\lambda(A)$ is the smallest
eigenvalue of $AA^{*}$, it holds that $1 \ge \lambda(A) \ge \det(AA^{*}).$ On the other hand, 
\[
\|A^{-1}\| = \sup_{\|v\| = 1} \langle A^{-1}v,A^{-1}v \rangle ^{1/2} \le \sup_{\|v\|=1} \langle v, (A^*)^{-1} A^{-1} v \rangle ^{1/2}  \le \lambda(A)^{-1/2}.
\]
Both imply that $\|A^{-1}\| \le \frac{1}{|\det(AA^*)|^{1/2}} \le \frac{1}{|\det(A)|},$ which implies our first claim. Now, we simply write 
\[
|x| = |A^{-1} \cdot A \cdot x| \le \|A^{-1}\| \cdot |A \cdot x| \le \|A\|^{n-1} \cdot |\det(A)|^{-1} \cdot |A \cdot x|,
\]
in order to conclude the proof.
\end{proof}
With this Lemma in hands, we simply need to notice that $\mathcal{Q}(s')$ equals 
\[
2^{j/2} \begin{pmatrix} 1 & 1 \\ (s' + 2^{-j}b(x)^2)^{-1} & h (hs' - \xi' + 2^{-j}b(y)^2)^{-1} \end{pmatrix} \cdot \begin{pmatrix} (s' + 2^{-j}b(x)^2)^{-1/4} \\ h^2 (hs' - \xi' + 2^{-j}b(y)^2)^{-1/4} \end{pmatrix}. 
\] 
\[
=: 2^{j/2} \mathcal{M}(s') \cdot \mathcal{V}(s').
\]
By the fact that $j \ge 0, b(x) \le 10,$ we see that $\|\mathcal{M}(s')\| \lesssim 1,$ as well as 
$$|\det(\mathcal{M}(s'))| \gtrsim \frac{|\xi'|}{|s' + 2^{-j}b(x)^2||hs' - \xi' + 2^{-j}b(x)^2|} \gtrsim 2^{-j/100}$$ 
and $|\mathcal{V}(s')| \gtrsim 1.$ Lemma \ref{matrix} gives that 
\[
|\mathcal{Q}(s')| \gtrsim 2^{j/3}. 
\]
Therefore, there is $i \in \{2,3\}$ so that $|\varphi^{(i)}(s')| \gtrsim 2^{j/3}.$ By Proposition 2 in Chapter VIII of \cite{Stein1}, we get that, for $|\xi'| \gtrsim 2^{-j/100},$ the integral from the Lemma is bounded by a multiple of 
\[
 \| \Psi(\xi',\cdot)\|_{C^2} \cdot 2^{-j/9}.
\]
This gives the second summand in the statement of Lemma \ref{oscilest0}, and therefore finishes the proof. 
\end{proof}
By Lemma \ref{oscilest0}, we obtain that $| (S^j_{N(x),b(x)}*\tilde{S}^j_{N(y),b(y)})(\xi)|$ is bounded by an absolute constant times
\[
 \frac{\mu(x,y)^2}{2^j} \left(1_{(-2^{-j/100},2^{-j/100})}\left(\frac{\mu(x,y)^2 \xi}{2^j}\right) + 2^{-j/9} 1_{(-4,4)}\left(\frac{\mu(x,y)^2 \xi}{2^j}\right) \right),
\]
where $\mu(x,y) = \min\{b(x),b(y)\}.$ Substituting into the formula of $\mathfrak{S}^j(\mathfrak{S}^j)^*f$, we obtain that for any $g \in L^2(\R),$ 
\[
|\langle \mathfrak{S}^j (\mathfrak{S}^j)^* f, g \rangle| \lesssim (2^{-j/100} + 2^{-j/9}) \cdot \left( \int_{\R} Mf(x) \cdot |g(x)| \mmd x + \int_{\R} Mg(x) |f(x)| \, \mmd x \right). 
\]
By $L^2$ boundedness of the maximal function, this is less than an absolute constant times $2^{-j/100}\|f\|_2 \|g\|_2.$ As a consequence, it follows that 
\[
\| \mathfrak{S}^j (\mathfrak{S}^j)^* f\|_2 \lesssim 2^{-j/100} \|f\|_2.
\]
Therefore, 
\[
\| \mathfrak{S}^j f\|_2 \lesssim 2^{-j/200} \|f\|_2.
\]
This finishes the proof of the proposition.
\end{proof}

\noindent\textbf{Part 2: Analysis of $\mathfrak{C}^2.$} For this part, we need  a slightly more sophisticated approximation to the phase function. We split the interval of integration as 
$$[-R,-2] \cup (-2,-1/2] \cup (-1/2,-b(x)^{-\frac{1}{6}}) \cup [-b(x)^{-\frac{1}{6}},b(x)^{-\frac{1}{6}}] \cup (b(x)^{-\frac{1}{6}},1/2) \cup [1/2,R].$$
We define the approximation polynomial 
$$P_{b(x)}(t) = b(x) \cdot \left( 1 + \frac{t}{2} - \frac{t^2}{4} + \frac{3t^3}{8} - \frac{15t^4}{16} + \frac{105t^5}{32}\right),$$
and note that 
\[
|b(x)[t+1]^{1/2} - P_{b(x)}(t)| \lesssim b(x) \cdot t^6
\]
for $t \in [-1/2,1/2].$ This follows directly by noting that $P_{b(x)}$ is nothing but the Taylor polynomial of order $5$ for the function $\sqrt{t+1}$ around the origin. 
Using this fact, we compare the integral defining our operator restricted to the middle interval:
\begin{equation}\label{compar}
\left| \int_{-b(x)^{-1/6}}^{b(x)^{-1/6}} f(x-t) e^{iN(x)t} e^{ib(x)[t+1]^{1/2}} \, \frac{\mmd t}{t} - \int_{-b(x)^{-1/6}}^{b(x)^{-1/6}} f(x-t) e^{iN(x)t} e^{iP_{b(x)}(t)} \, \frac{\mmd t}{t}\right| 
\end{equation}
is bounded by 
\[
C b(x) \cdot \int_{-b(x)^{-1/6}}^{b(x)^{-1/6}} |f(x-t)| t^5 \, \mmd t \le C b(x)^{1/6} \int_{-b(x)^{-1/6}}^{b(x)^{-1/6}} |f(x-t)| \, \mmd t \lesssim Mf(x).
\]
We have to resort to the full version of the polynomial Carleson theorem in \cite{Lie2} and \cite{pavel} to bound the second integral in \eqref{compar} in $L^p$: the operator to which we compared is bounded by a maximally truncated polynomial Carleson operator 
of degree $\le 5,$ which is $L^p-$bounded by the aforementioned references. As the difference is bounded in $L^p$, the analysis for $[-b(x)^{-1/6},b(x)^{-1/6}]$ is complete. \\

It remains to bound the operators relative to the integral over the ``outer" layers. We first begin by analyzing the outermost intervals
$$(-1/2,-b(x)^{-1/6}) \cup (b(x)^{-1/6},1/2).$$
As the proof for both of them is essentially the same, we focus on the positive interval $(b(x)^{-1/6},1/2).$ \\ 

Let $\psi_0$ be a smooth bump with the properties as in Part 1. Up to a maximal function error, bounding the integral defining $\mathfrak{C}^2$ over the interval $(b(x)^{-1/6},1/2)$ 
is equivalent to bounding 
\[
\tilde{\mathfrak{C}}f(x):=\int f(x-t) e^{iN(x)t} e^{ib(x)[t+1]^{1/2}} \cdot \phi_{b(x)}(t)\,  \frac{\mmd t}{t},
\]
where $\phi_{b(x)}(t) = \phi_{\lfloor \log_2 b(x) \rfloor}(t) := \sum_{j = (2-\frac{1}{6})\lfloor \log_2 b(x) \rfloor}^{2  \lfloor \log_2 b(x) \rfloor - 3} \psi_0(2^{-j}2^{2 \lfloor \log_2 b(x) \rfloor} t).$ This holds by the fact that the smooth cutoff function 
$\phi_{b(x)}$ approximates well the characteristic function of $(b(x)^{-1/6},1/2)$ due to the properties of $\psi_0.$ \\

Our main goal now is to achieve exponential decay in $\lfloor \log_2 b(x) \rfloor$. This will be enough for our purposes, as the operator $\tilde{C}$ behaves well in the sets where $b \sim 2^k.$ Explicitly, we compute: 
\begin{align*}
|\tilde{\mathfrak{C}}f(x)| & \le \left(\sum_{k \ge 3}  \left| 1_{b_k(x) \in (1,2]} \int f(x-t)e^{iN(x)t} e^{i2^k b_k(x)[t+1]^{1/2}} \phi_k(t) \, \frac{\mmd t}{t} \right|^p \right)^{1/p} \cr
 & =: \left( \sum_{k \ge 3} |\mathfrak{C}_{k}f(x)|^p\right)^{1/p}, \cr 
\end{align*}
where $b_k(x) = \frac{b(x)}{2^k}$. It suffices then to bound $\|\mathfrak{C}_{k}f\|_p \lesssim k\cdot  2^{- \alpha_p \cdot k} \|f\|_p,$ with $\alpha_p > 0.$ In order to obtain 
that bound, we decompose each of the $\mathfrak{C}_{k}$ further as 
\[
\mathfrak{C}_{k}f(x) = \sum_{j= \frac{11 k}{6}}^{2k} \mathfrak{C}_{k}^jf(x),
\]
where
\[
\mathfrak{C}_{k}^jf(x) := 1_{b_k(x) \in (1,2]} \int f(x-t)e^{iN(x)t} e^{i2^k b_k(x)[t+1]^{1/2}} \psi_0(2^{j-2k} t) \, \frac{\mmd t}{t}.
\]
Our main Proposition to get decay in $j$ for these operators reads as follows. 

\begin{prop}\label{doubledecay} There exists $\beta > 0$ such that for all $k \ge 1$ and all $j \in ((2-\frac{1}{6})k,2k),$
\[
\|\mathfrak{C}_{k}^j f\|_2 \lesssim \left( 2^{-\beta j} + 2^{\beta \cdot( \frac{7}{3}k - \frac{399}{300} j)} \right) \|f\|_2. 
\]
\end{prop}

It is direct to notice that the pointwise bound 
\[
|\mathfrak{C}_{k}^jf(x)| \lesssim Mf(x)
\]
holds independently of $j \in [11k/6,2k].$ This estimate implies automatically the endpoint 
\begin{equation}\label{endpoint}
\|\mathfrak{C}_{k}^jf\|_{\infty} \lesssim \|f\|_{\infty}, \,\, \|\mathfrak{C}_{k}^jf\|_{1,\infty} \lesssim \|f\|_1
\end{equation}
bounds, so that interpolating between Proposition \ref{doubledecay} and estimate \eqref{endpoint} gives us the existence of $\beta_p > 0$ such that for all $k\ge 1$ and all $j \in [11k/6,2k]$, 
\[
\|\mathfrak{C}_{k}^jf\|_p \lesssim \left( 2^{-\beta_p j} + 2^{\beta_p \cdot( \frac{7}{3}k - \frac{399}{300} j)} \right) \|f\|_p. 
\]
\begin{proof}[Proof of Proposition \ref{doubledecay}] In order to get decay on $\mathfrak{C}_{k}^jf := \Phi_{N(x),b_k(x)}^{k,j} * f(x)$ in $L^2,$ it suffices to bound $\mathfrak{C}_{k}^j(\mathfrak{C}_{k}^j)^{*}f$ instead. A computation gives us that 
\[
\mathfrak{C}_{k}^j(\mathfrak{C}_{k}^j)^* f(x) = 1_{b_k(x) \in (1,2]} \int_{\R} (\Phi_{N(y),b_k(y)}^{k,j}*\tilde{\Phi}_{N(x),b_k(x)}^{k,j})(x-y) (1_{b_k(y) \in (1,2]}f)(y)\, \mmd y,
\]
with $\tilde{\Phi}_{N(\cdot),b_k(\cdot)}^{k,j}(\xi) := \overline{\Phi_{N(\cdot),b_k(\cdot)}^{k,j}} (-\xi).$ The convolution inside this integral is explicitly given by 
\[
\int_{\R} e^{i(N(x) - N(y))s} \cdot e^{2^k \cdot i (b_k(x)\sqrt{s+1} - b_k(y)\sqrt{s-\xi + 1})} \cdot \frac{\psi_0(2^{2k-j}s)}{s} \frac{\psi_0(2^{2k-j}(s-\xi))}{s- \xi} \, \mmd s,
\]
times a modulating factor depending on $\xi$ but not on $s.$ As our goal is to bound the absolute value of this expression, we can safely ignore it. In order to bound this expression, 
we assume, without loss of generality, that $b_k(y) \le b_k(x).$ By changing variables $s = 2^{j-2k} b_k(x)^{-2} s'$ and letting $\xi = 2^{j-2k} b_k(y)^{-2} \xi',$ we rewrite it as 
\[
\frac{b_k(y)^2}{2^{j-2k}} \cdot \int_{\R} e^{i(\tilde{N}(x) - \tilde{N}(y))s'} \cdot e^{i \tilde{R}_{\xi',j,k}(s')}\frac{ \psi_0(b_k(x)^{-2}s')}{s'} \frac{\psi_0(b_k(y)^{-2} (hs' - \xi '))}{hs' - \xi'} \, \mmd s',
\]
where 
$$\tilde{R}_{\xi',j,k}(s') = 2^{j/2} \cdot (\sqrt{s' + 2^{2k-j}b_k(x)^2} - \sqrt{h s' - \xi' + 2^{2k-j}b_k(y)^2}),$$ 
$\tilde{N}$ is a measurable function and $h:= \left(\frac{b_k(y)}{b_k(x)}\right)^2 \le 1$. Notice that the function 
$$\frac{ \psi_0(b_k(x)^{-2}s')}{s'} \frac{\psi_0(b_k(y)^{-2} (hs' - \xi '))}{hs' - \xi'}$$
is smooth, bounded and supported in $s' \in [1/4,4]$ with bounded $C^3$ norm, as $b_k(x),b_k(y) \in (1,2).$ This allows us to focus on the oscillatory nature of the phase. \\ 

The next lemma is the tool we need to bound this kernel pointwise.

\begin{lemma}\label{oscilest} Let $\Psi: \R \times \R \to \R$ be a smooth function supported in $\{ (s',\xi') \in \R^2 \colon s', hs' - \xi' \in [1/4,4]\},$ for some fixed positive parameter $h \le 1.$ It holds that, for all $v, \xi' \in \R$ and $j \in [11k/6,2k],$ 
\begin{align}\label{oscil}
 & \left| \int_{\R} e^{i v s'} \cdot e^{i \tilde{R}_{\xi',j,k}(s')}  \Psi(\xi',s') \mmd s' \right|   \cr
 \lesssim \sup_{\xi' \in [-4,4]} \|\Psi(\xi',\cdot)\|_{C^2} & \left( 1_{[-2^{-j/100},2^{-j/100}]}(\xi') +  2^{\frac{7}{3}k-\frac{399}{300}j} 1_{[-4,4]}(\xi')\right), \cr
\end{align}
where the implicit constant does not depend on $h \in (0,1].$ 
\end{lemma} 

\begin{proof}[Proof of Lemma \ref{oscilest}] This is an application of the stationary phase principle: \\ 

\noindent If $|\xi'| \le 2^{-\frac{1}{100} j},$ we bound the integral by 
taking the modulus inside, and we get the first summand $1_{[-2^{-\frac{j}{100}},2^{-\frac{j}{100}}]}(\xi')$ on the right hand side. \\ 

If, on the other hand, $|\xi'| \ge 2^{-\frac{j}{100}},$ then we denote for shortness 
$\phi(s') = vs' + \tilde{R}_{\xi',j,k}(s').$ We consider the vector 
$$ Q(s') = \begin{pmatrix} \phi''(s') \\ -\frac{2}{3} \phi'''(s') \end{pmatrix}.$$
Our aim is to prove that the norm $|Q(s')| \gtrsim 2^{\frac{399}{300}j - \frac{7}{3}k} $, as this implies that either the second or the third derivative of the phase $\phi(s')$ have this same, what enables us then to use stationary phase to conclude the proof. 
For that purpose, we write the vector $Q(s')$ alternatively as 
\[
\frac{2^{j/2}}{4} \cdot \begin{pmatrix} 1 & 1 \\ (s' + 2^{2k-j}b_k(x)^2)^{-1} & h(hs' - \xi' + 2^{2k-j}b_k(y)^2)^{-1} \end{pmatrix} \cdot V(s') = \frac{2^{j/2}}{4} \cdot M(s') \cdot V(s'),
\]
where $V(s') = \begin{pmatrix} (s' + 2^{2k-j}b_k(x)^2)^{-\frac{3}{2}} \\ (hs' - \xi' + 2^{2k-j}b_k(y)^2)^{-\frac{3}{2}} \end{pmatrix}.$ By the fact that $s',hs'-\xi' \in [1/4,4],$ we see that 
\[
|\det(M(s'))| \gtrsim \frac{|\xi'|}{2^{4k-2j}} \ge 2^{(2-\frac{1}{100})j - 4k}.
\]
It is also straightforward to see that the supremum norm $\|M(s')\| \lesssim 1$. By Lemma \ref{matrix}, the representation formula for $Q(s')$ and the fact that $|V(s')| \gtrsim 2^{\frac{3j}{2} - 3k}$, it holds that  
\[
|Q(s')| \gtrsim 2^{j/2} \cdot 2^{(2-\frac{1}{100})j - 4k} \cdot 2^{(\frac{3j}{2}-3k)} = 2^{(4-\frac{1}{100})j - 7k}.
\]
As this implies that either the second or third derivatives of the function $\phi(s')$ above are bounded from below by $2^{\frac{399}{100}j - 7k}$. By the stationary 
phase principle as stated in \cite[Proposition~VIII.2]{Stein1}, the oscillatory integral 
\[
\left| \int_{\R} e^{i\phi(s')} \Psi(\xi',s') \, \mmd s' \right| \lesssim \| \Psi(\xi',\cdot)\|_{C^2} 2^{\frac{7}{3}k-\frac{399}{300}j},
\]
whenever $|\xi'| \ge 2^{-j/100}.$ This gives us the second summand in the statement of Lemma \ref{oscilest}, and therefore the result.
\end{proof}

In order to finish the proof of the proposition, we notice that the oscillatory kernel given from the convolution defining the kernel of $\mathfrak{C}_{k}^j(\mathfrak{C}_{k}^j)^*$ fits the framework of Lemma \ref{oscilest}. Therefore, by using that 
$b_k(y),b_k(x) \in (1,2],$ we obtain 
\begin{align*}
 & |\mathfrak{C}_{k}^j(\mathfrak{C}_{k}^j)^*f(x)| \lesssim \frac{1}{2^{j-2k}} \int_{-4\cdot 2^{\frac{99j}{100}-2k}}^{4 \cdot 2^{\frac{99j}{100} -2k}} |f(x-y)| \, \mmd y \cr
 & + \frac{2^{(\frac{7k}{3} - \frac{399}{300}j)}}{2^{j-2k}} \int_{-10\cdot 2^{j-2k}}^{10 \cdot 2^{j-2k}} |f(x-y)| \, \mmd y \lesssim 2^{-j/100} Mf(x) + 2^{(\frac{7k}{3} - \frac{399}{300}j)} Mf(x).  \cr
\end{align*}
In particular, by boundedness of the maximal function,
\[
\|\mathfrak{C}_{k}^j(\mathfrak{C}_{k}^j)^*f\|_2 \lesssim \left(2^{-j/100} + 2^{(\frac{7k}{3} - \frac{399}{300}j)} \right) \|f\|_2. 
\]
This implies directly that
\[
\|\mathfrak{C}_{k}^jf\|_2 \lesssim \left(2^{-j/200} + 2^{\frac{1}{2}\left(\frac{7k}{3} - \frac{399}{300}j\right)}\right) \|f\|_2.
\]
This finishes the proof of the proposition by taking $\beta = \frac{1}{200}.$  
\end{proof}
With the proposition in hands, our previous considerations yield that there is $\beta_p > 0$ so that 
\[ 
\|\mathfrak{C}_{k}^jf\|_p \lesssim_p \left(2^{-\beta_p \cdot j} + 2^{\beta_p \cdot \left(\frac{7k}{3} - \frac{399}{300}j\right)}  \right) \|f\|_p.
\]
Summing for $j \in [11k/6,2k]$ yields 
\begin{align*}
\|\mathfrak{C}_{k}f\|_p \le & \sum_{j = 11k/6}^{2k} \|\mathfrak{C}_{k}^jf\|_p  \lesssim_p \sum_{j = 11k/6}^{2k} \left(2^{-\beta_p \cdot j} + 2^{\beta_p \cdot \left(\frac{7k}{3} - \frac{399}{300}j\right)}  \right) \|f\|_p \cr
 & \lesssim_p k \cdot (2^{-\frac{11\beta_p}{6} k} + 2^{\beta_p \cdot \left(\frac{7k}{3} - \frac{4389}{1800}k\right)}) \|f\|_p \lesssim_p k\cdot 2^{-\alpha_p k}\|f\|_p, \cr
 \end{align*}
for $\alpha_p = \frac{189}{1800} \cdot \beta_p.$ This implies, on the other hand, that
\[
\|\tilde{\mathfrak{C}}f\|_{p} \le \left(\sum_{k \ge 3} \|\mathfrak{C}_{k}f\|_{p}^p \right)^{1/p} \lesssim \left(\sum_{k \ge 3} k^p \cdot 2^{-p\alpha_p k}\right)^{1/p} \|f\|_p \lesssim \|f\|_p,
\]
which concludes the proof of boundedness for the intervals $(-1/2,-b(x)^{-1/6})\cup (b(x)^{-1/6},1/2).$ \\ 

We briefly remark on the necessity of a large degree approximation of the Taylor polynomial in the phase. Indeed, redoing the argument above shows that choosing a Taylor polynomial of 
degree $d$ splits naturally the integration interval as $(-1/2,-b(x)^{-1/(d+1)}) \cup [-b(x)^{-1/(d+1)},b(x)^{-1/(d+1)}] \cup (b(x)^{-1/(d+1)},1/2).$ In the middle interval, the same comparison holds as before, 
whereas the bounds for each scale in the outer intervals are \emph{not} altered. \\ 

That is, we still obtain a $2^{\beta\left(\frac{7k}{3} - \frac{399j}{300}\right)}$ factor, which we wish to decay 
exponentially with $k>0.$ For that, we must have, necessarily, $j>\frac{7}{4}k.$ But from our definitions, in the intervals $(b(x)^{-1/(d+1)},1/2),$ we obtain $j - 2k > \frac{-k}{d+1} \iff j > \frac{(2d+1)k}{d+1}.$ 
In order for this last factor to be $> \frac{7}{4},$ we need $d \ge 4.$ The choice $d=5$ comes about in order to relax the tightness in the various steps of the proof. \\ 

For the interval $[-2,-1/2]$ where the singularity of the phase $[t+1]^{1/2}$ lies, one notices that the kernel $\frac{1}{t}$ has upper and lower bounds, so that the integral 
\[
\left| \int_{-2}^{-1/2} f(x-t) e^{iN(x)t} e^{ib(x)[t+1]^{1/2}} \, \frac{\mmd t}{t} \right| \lesssim \int_{-2}^{-1/2} |f(x-t)| \, \mmd t \lesssim Mf(x).
\]
We are then left with the outermost intervals $[-R,-2] \cup [1/2,R].$ As the analysis is virtually the same in both cases -- and as the phase $b(x) \cdot [t+1]^{1/2}$ does not change sign in either of the intervals --,
we focus on the positive interval $[1/2,R].$ \\ 

We decompose the operator $\mathfrak{C}^2$ directly this time, without the need to identify the scale of $b.$ Explicitly, we write the integral
\[
\int_{1/2}^R f(x-t) e^{iN(x)t} e^{ib(x)[t+1]^{1/2}} \, \frac{\mmd t}{t},
\]
modulo error terms that amount to a constant times the Hardy--Littlewood maximal function of $f$, as 

\begin{equation}\label{bounds}
\sum_{j \ge 2\log_2 b(x)-2}^{\log_2 R} \int_{\R} f(x-t) e^{iN(x)t} e^{ib(x)[t+1]^{1/2}} \psi_0(2^{-j}b(x)^2t) \, \frac{\mmd t}{t} = \sum_{j \ge 1} \mathfrak{T}^jf(x).
\end{equation}
Notice that we encompass the fact that $j \ge 2 \log_2 b(x) - 2$ already in the definition of $\mathfrak{T}^j.$ 
We must now only prove exponential decay in $j$ in the $L^p$ norms of $\mathfrak{T}^jf.$ The proof of this fact follows essentially the same line as before: after all the reductions, we need to look at an oscillatory integral representing the kernel of $\mathfrak{T}^j (\mathfrak{T}^j)^*$.
This is given by a similar oscillatory integral to the one before, i.e., 
\[
\frac{b(y)^2}{2^j} \int_{\R} e^{i(\tilde{N}(x) - \tilde{N}(y))s'} \cdot e^{i \tilde{R}_{\xi',j,k}(s')}\frac{ \psi_0(s')}{s'} \frac{\psi_0(hs' - \xi ')}{hs' - \xi'} \, \mmd s',
\]
where 
$$\tilde{\tilde{R}}_{\xi',j,k}(s') = 2^{j/2} \cdot (\sqrt{s' + 2^{-j}b(x)^2} - \sqrt{h s' - \xi' + 2^{-j}b(y)^2}),$$ 
$\tilde{N}$ is a measurable function and $h:= \left(\frac{b(y)}{b(x)}\right)^2 \le 1$. What changes now are the estimates we can achieve with the stationary phase method. Now, the vector 
\[
\tilde{Q}(s') = \begin{pmatrix} \tilde{\phi}''(s') \\ -\frac{2}{3}\tilde{\phi}'''(s') \end{pmatrix} = 2^{j/2} \tilde{M}(s') \cdot \tilde{V}(s') 
\]
has slightly different properties: it is easy to see that 
$$\tilde{M}(s') = \begin{pmatrix} 1 & 1 \\ (s' + 2^{-j}b(x)^2)^{-1} & h(hs' - \xi' + 2^{-j}b(y)^2)^{-1} \end{pmatrix}, $$
and therefore $\|\tilde{M}(s')\| \lesssim 1$ still, but now, as $b(x)^2 2^{-j} \lesssim 1,$ the determinant bounds change to 
\[
|\det (\tilde{M}(s'))| \gtrsim |\xi'| \ge 2^{-j/100}. 
\]
Also, we can only ensure that $|\tilde{V}(s')| \gtrsim 1.$ This implies, by Lemma \ref{matrix}, that $|\tilde{Q}(s')| \gtrsim 2^{j/3}.$ Stationary phase and the considerations as in Lemma \ref{oscilest} give the bound
\[
\|\mathfrak{T}^jf\|_2 \lesssim 2^{-j/200} \|f\|_2.
\]
It is direct to conclude from the definition that $|\mathfrak{T}^jf| \lesssim Mf(x),$ so that interpolation gives the existence of $\theta_p>0$ so that $\|\mathfrak{T}^j f\|_p \lesssim 2^{-\theta_p j} \|f\|_p.$ as $b(x) \ge 10,$ we see that $j \ge 1,$ so that 
the $L^p$ norm of the sum in \eqref{bounds} is pointwise bounded by 
\[
\sum_{j \ge 1} 2^{-\theta_p j} \|f\|_p \lesssim_p \|f\|_p. 
\]
This concludes the analysis of $L^p$ bounds of $\mathfrak{C}^2$ and therefore the proof of Theorem \ref{main}. $\square$ \\

As mentioned in the introduction, one wonders whether the analysis for the intervals $(-1/2,-b(x)^{-1/6}) \cup (b(x)^{-1/6},1/2)$ in the proof of Theorem \ref{main} can be suppressed by using a better polynomial approximation. The next proposition proves that employing
our approximation technique is impossible without being forced to allow the degree of the polyonomial to depend on $b$:

\begin{prop}\label{polyinfty} Suppose that, for each $b \ge 1,$ we are given a polynomial $P_b(t)$ such that 
\[
\int_{-1/2}^{1/2} |b\sqrt{t+1} - P_b(t)| \, \frac{\mmd t}{|t|} \le 1.
\]
Then $\lim_{b \to \infty} \text{deg }(P_b) = + \infty.$ 
\end{prop}

\begin{proof}[Proof of Proposition \ref{polyinfty}] In order for the integral 
\[
\int_{-1/2}^{1/2} |b\sqrt{t+1} - P_b(t)| \, \frac{\mmd t}{|t|}
\]
to be finite, we must have that $P_b(0) = b.$ The condition on the polynomials given by the proposition then becomes 
\[
\int_{-1/2}^{1/2} \left|(\sqrt{t+1} - 1) - \left(\frac{P_b(t)}{b}-1\right)\right| \frac{\mmd t}{|t|} \le \frac{1}{b}. 
\]
The last inequality reveals that (a) $\frac{P_b(t)}{b} - 1 \to \sqrt{t+1} - 1$ in $L^1([-1/2,1/2],\frac{\mmd t}{|t|});$ (b) The sequence $\frac{\left(\frac{P_b(t)}{b} - 1 \right)}{t}$ is bounded in $L^1(-1/2,1/2).$ \\ 

Now we suppose that there is an upper bound on the degrees of the polynomials $P_b.$ As the sequence $\frac{\left(\frac{P_b(t)}{b} - 1 \right)}{t}$ lies in a finite-dimensional polynomial space, 
all norms are equivalent. In particular, the sum of coefficients norm is bounded by the $L^1(-1/2,1/2)$ norm. This and (b) give us that, denoting this norm by $\| \cdot\|_{\text{coeff}},$ 
\[
\left\| \frac{\left(\frac{P_b(t)}{b} - 1 \right)}{t} \right\|_{\text{coeff}} \le C, \, \forall b \ge 1.
\]
But, from (a), we know that $\frac{\left(\frac{P_b(t)}{b} - 1 \right)}{t} \to (\sqrt{t+1} -1)/t$ in $L^1(-1/2,1/2).$ We then extract a subsequence of $\{b_k\}_k$ so that (a) $\frac{\left(\frac{P_{b_k}(t)}{b_k} - 1 \right)}{t} \to \frac{\sqrt{t+1}-1}{t}$
\emph{almost everywhere} in $(-1/2,1/2);$ and (b) the coefficients of $\frac{\left(\frac{P_{b_k}(t)}{b_k} - 1 \right)}{t}$ converge. But this is already a contradiction, for then, as the degree is bounded and coefficients converge, 
$\frac{\left(\frac{P_{b_k}(t)}{b_k} - 1 \right)}{t}$ should converge to a polynomial pointwise, which is clearly not the case for $\frac{\sqrt{t+1}-1}{t}.$ 
\end{proof}  

\section{Comments and remarks}\label{comments}

\subsection{Question \ref{parabcarl} and Theorem \ref{lines}} As briefly sketched in the introduction, if the answer to Question \ref{parabcarl} is \emph{affirmative}, then Theorem \ref{lines} holds. We explain this relationship in greater detail here. As discussed before, we let $\ell$ be a line in $\R^2$. Without loss of generality, we 
suppose it is given by an equation of the form $(a \eta + b,\eta),$ as the remaining case of having equation $(\eta,C),$ $C$ constant, is completely analogous. Let $\mathcal{C}_{[\ell]} = \sup_{b \in \R} \mathcal{C}_{a,b}$ be the operator associated to the equivalence class of lines $\ell' \sim \ell$, where two lines are equivalent if they have 
the same slope. \\ 

We fix $g \in \mathcal{S}(\R)$ and write $\xi = \theta \cdot v_{\ell} + t \cdot v_{\ell}^{\perp},$ $\,(\theta,t) \in \R^2$ and $v_{\ell} = (a,1).$  If we fix $N_{(\cdot,\cdot)}:\R^2 \to \R$ a measurable function and $F_{\delta}$ such that $\widehat{F_{\delta}}(\theta \cdot v_{\ell} + t \cdot v_{\ell}^{\perp}) = \widehat{g}(\theta) \psi_{\delta}(t)$, we have 
\[
\mathcal{C}_2 (F_{\delta})(z) \ge \left| \int_{\R} \left( \int_{\R} \widehat{g}(\theta) e^{2 \pi i \theta \langle z,v_{\ell} \rangle} m_2(\theta \cdot v_{\ell} + t \cdot v_{\ell}^{\perp} + N_z) \, \mmd \theta \right) e^{2 \pi i t \langle z, v_{\ell}^{\perp} \rangle} \psi_{\delta}(t) \, \mmd t \right|.
\]
Choose $\psi_{\delta} = \frac{1}{\delta^{1/2}} \varphi\left(\frac{x}{\delta}\right),$ with $1_{[-1,1]} \le \varphi \le 1_{[-2,2]}$ smooth. Let 
\[
G_{[\ell]}(z,t) =  \int_{\R} \widehat{g}(\theta) e^{2 \pi i \theta \langle z,v_{\ell} \rangle} m_2(\theta \cdot v_{\ell} + t \cdot v_{\ell}^{\perp} + N_z) \, \mmd \theta. 
\]
Since $g \in \mathcal{S}(\R),$ it is easy to see by the dominated convergence theorem that $G_{[\ell]}(z,t) \to G_{[\ell]}(z) :=  \int_{\R} \widehat{g}(\theta) e^{2 \pi i \theta \langle z,v_{\ell} \rangle} m(\theta \cdot v_{\ell} + N_z) \, \mmd \theta$ pointwise, as $m_2$ is bounded and continuous in $\R^2.$ By choosing $N$ suitably, we can make $|G_{[\ell]}(z)| \ge \frac{1}{2} C_{[\ell]}g(\langle z, v_{\ell} \rangle), \, \forall z \in \R^2.$ 
Reasoning again with dominated convergence gives  
\[ 
\left|\int_{\R} G_{[\ell]}(z,t)e^{2 \pi i t \langle z, v_{\ell}^{\perp} \rangle} \frac{\psi_{\delta}(t)}{\delta^{1/2}} \, \mmd t - \int_{\R} G_{[\ell]}(z)e^{2 \pi i t \langle z, v_{\ell}^{\perp} \rangle} \frac{\psi_{\delta}(t)}{\delta^{1/2}} \, \mmd t\right| \to 0
\]
as $\delta \to 0.$ Moreover, each of the integrals above is bounded as a function of $z.$ These considerations imply that, for $R > 0$ fixed, 
\begin{equation}
\|\mathcal{C}_2 (F_{\delta}) \|_{L^2(B_R)} + o^R_{\delta}(1) \ge \frac{1}{2} \| \delta^{1/2} \widehat{\varphi}(\delta \langle \cdot, v_{\ell}^{\perp} \rangle) C_{[\ell]}g(\langle \cdot, v_{\ell}\rangle)\|_{L^2(B_R)} \ge \frac{1}{5} \|C_{[\ell]}g\|_{L^2\left(-\frac{R}{2},\frac{R}{2}\right)}.
\end{equation}
We use here $o_{\delta}^R(1)$ to denote a quantity that goes to $0$ as $\delta \to 0,$ with $R$ fixed. But, assuming the answer of Question \ref{parabcarl} to be affirmative, 
\[
\|\mathcal{C}_2(F_{\delta})\|_{L^2(B_R)} \le C \|F_{\delta}\|_{L^2(\R^2)} = C\|\widehat{F_{\delta}}\|_{L^2(\R^2)}.
\]
By using the explicit representation of $F_{\delta}$ and the choice of $\psi,$ we get that the right hand side converges to $C \|g\|_{L^2(\R)}$ as $\delta \to 0.$ Putting together, we have 
\[
\|C_{[\ell]} g\|_{L^2\left(-\frac{R}{2},\frac{R}{2}\right)} \le 5C \cdot \|g\|_{L^2(\R)}.
\]
Notice that $C$ is independent of both $R, [\ell].$ By taking $R \to \infty$ in this last inequality one obtains Theorem \ref{lines}. 

\subsection{Hilbert transform along more general curves} Throughout this article, we have investigated the case of the Hilbert transform along the parabola $(t,t^2).$ There is, however, no reason not to consider more general \emph{monomial} curves of the form $(t,t^m).$ For those, it is natural to expect that the reductions performed 
in Section \ref{reduction} carry through, and that effectively one needs to bound the operator 
\[
f \mapsto \sup_{N,b} \left| \int_{-1/2}^{1/2} f(x-t) e^{iNt} \cdot e^{ib[t+1]^{1/m}} \, \frac{\mmd t}{t}  \right|,
\]
where $[u]^r$ represents either $|u|^r$ or $\text{sign}(u)|u|^r.$ The proof in Section \ref{modeluni} is not particular to the $m=2$ case, and therefore can be adapted to prove that these operators are bounded in $L^p.$ The reduction to these operators is not as direct as the quadratic case, though. For the case of higher degrees, one would have to use a form of decomposition 
as in the recent article by Guo \cite{Guo1}. Following this idea, it should be possible to exploit the \emph{polynomial} case $(t,Q(t))$, $Q \in \text{Poly}(\R\colon \R)$. In order to keep the exposition short, we do not investigate these questions further. 

\subsection{Oscillatory integrals and the proof of Theorem \ref{lines}} The proof of boundedness of the operators $\mathfrak{C}^R$ highlights what seems to be a general principle: if we are given a function $\eta: \R \to \R$ whose derivative is singular in a neighbourhood of the origin, but sufficiently regular (together with its higher degree derivatives) everywhere else, then the maximal function 
\begin{equation}\label{multipoly}
f \mapsto \sup_{\substack{{P \in \mathcal{P}_d} \\{N \in \R}}} \left| \int_{\R} f(x-t) e^{iP(t) + iN \cdot \eta(t+1)} \, \frac{\mmd t}{t} \right|,
\end{equation}
where $\mathcal{P}_d$ is the space of polynomials of degrees $\le d,$ should be bounded in $L^2.$ In this article, we have explored the case $d=1$, where $\eta(t) = |t|^{1/2}$ or $\eta(t) = \text{sign}(t)|t|^{1/2}.$  We notice, however, that by taking further derivatives of the phase, defining approximation polynomials 
with higher degrees and running the basic strategy we set here, there should not stand any barrier to prove the case of general $d>1.$ This suggests the existence of an underlying principle for a more general class of functions $\eta$ whose decay are sufficiently controllable. We currently believe this principle is intimately 
related to Question \ref{parabcarl}.

\section{Ackowledgements}

The author is indebted to his doctoral advisor Prof. Dr. Christoph Thiele for having suggested investigating Theorem \ref{lines}. He also thanks Joris Roos, for helpful discussions, in both
early and late stages of this project, and Pavel Zorin-Kranich for discussions on the main results of this manuscript and insights on his article \cite{pavel}. Finally, the author ackowledges finantial support by the Deutscher Akademischer Austauschdienst (DAAD).

\end{document}